\documentclass[11pt, reqno]{amsart}
\allowdisplaybreaks[3]

\usepackage[indentafter]{titlesec}
\usepackage{enumerate}
\usepackage{mathrsfs}
\usepackage{amsopn,amscd,amsmath,amssymb,amsthm,amsfonts,epsfig,graphics}
\usepackage{pslatex,avant,color}
\usepackage{cite}

\theoremstyle{plain}
\newtheorem{thm}{Theorem}[section]

\newtheorem{cor}{Corollary}[section]
\newtheorem{propo}{Proposition}[section]
\newtheorem{defi}{Definition}[section]
\newtheorem{exm}{Example}[section]

\theoremstyle{remark}
\newtheorem{rem}{Remark}[section]

\titleformat{name=\section}{}{\thetitle.}{0.8em}{\centering\scshape}
\titleformat{name=\subsection}{}{\thetitle.}{0.5em}{\bfseries}[]

\numberwithin{equation}{section}
\begin{document}

\title[Construction of Pole Cancellation Functions]{Construction of Pole Cancellation Functions at Ordinary Poles of Operator-Valued Functions}
\author{Muhamed Borogovac}

\begin{abstract} 
A pole of order $m \in \mathbb{N}$ at $\beta \in \mathbb{C}$ of a regular operator valued function $Q : \mathcal{D}(Q) \to \mathcal{L}(\mathcal{H})$ is investigated. We provide a characterization of pole cancellation functions $\boldsymbol{\psi}(z)$ of $Q(z)$ of order $k \le m$ at $\beta$ in terms of the coefficients of the Laurent expansion of $Q$. This characterization yields practical and explicit constructions of pole cancellation functions $\boldsymbol{\psi}(z)$. Moreover, it leads to an explicit formula for the associated functions $\boldsymbol{\hat{\varphi}}(z) := Q(z)\boldsymbol{\psi}(z)$, which are root functions of order $k$ at the zero $\beta$ of $Q^{-1}$. The results are illustrated by an example.
\end{abstract}


\maketitle
\thispagestyle{empty}
\section{Introduction}\label{s2}

\textbf{1.1.}  Let $\mathbb{N}$, $\mathbb{R}$, and $\mathbb{C}$ denote the sets of positive integers, real numbers, and complex numbers, respectively. The extended complex plane is defined as $\bar{\mathbb{C}} := \mathbb{C} \cup \{\infty\}$. Let $\mathcal{H}$ denote a Hilbert space, and let $\mathcal{L}(\mathcal{H})$ denote the Banach space of bounded linear operators on $\mathcal{H}$. 

We assume that $\mathcal{D}(Q) \subseteq \mathbb{C}$ is a domain, and that $Q : \mathcal{D}(Q) \to \mathcal{L}(\mathcal{H})$ is a holomorphic operator-valued function on $\mathcal{D}(Q)$, that is, the domain of $Q$ coincides with its domain of holomorphy. We study poles of such functions. Obviously, these functions may have other kinds of singularities outside their domains of meromorphy, such as generalized poles of generalized Nevanlinna-Herglotz functions; see, for example, \cite{Lu3}. For the study of generalized poles of generalized Nevanlinna-Herglotz
functions, operator and integral representations of these functions are required, which are not available in the general case of operator-valued functions.

If an operator-valued function $Q(z)$ is boundedly invertible at least at one point of holomorphy $w \in \mathcal{D}(Q)$, then it is boundedly invertible on some domain in $\mathbb{C}$ and is called a \textit{regular} operator-valued function on that domain; see, e.g., \cite[p.~327]{Lu4}. If $\mathcal{H} = \mathbb{C}^{n}$, then we deal with matrix-valued functions. Evidently, a matrix-valued function is regular if and only if $\chi(z) := \det Q(z) \not\equiv 0$.

Recall that an isolated singularity $\beta \in \mathbb{C}$ is called a \textit{pole of order} $m \in \mathbb{N}$ of the operator-valued function $Q(z)$ if $m$ is the minimal number such that the function $\tilde{Q}(z) := (z-\beta)^{m} Q(z)$ is holomorphic at $\beta$. Then $Q(z)$ admits a Laurent expansion
\[
Q(z) = \sum_{j = -m}^{\infty} Q_{j} (z- \beta)^{j},
\]
in some neighborhood of $\beta$ excluding $\beta$, where each coefficient $Q_j$ is a bounded linear operator, that is, it belongs to $\mathcal{L}(\mathcal{H})$. Explicitly,
\begin{equation}
\label{eq22}
Q(z) = \frac{Q_{-m}}{(z- \beta)^{m}} + \cdots + \frac{Q_{-1}}{(z- \beta)} + Q_{0} + Q_{1}(z- \beta) + Q_{2}(z- \beta)^{2} + \cdots ,
\end{equation}
where the fractional part is called the \textit{principal part of the Laurent expansion}, $Q_{-m}$ is called the \textit{leading principal coefficient}, and $Q_{-1}$ is called the \textit{residue operator}. The coefficients of the expansion \eqref{eq22} are easily derived using the Taylor expansion of the following function, which is holomorphic at $\beta$:
\[
\tilde{Q}(z) := (z- \beta)^{m} Q(z) = \sum_{j = -m}^{\infty} Q_{j} (z- \beta)^{m+j}.
\]
After the substitution $i = m + j$, we obtain
\begin{equation}
\label{eq24}
\tilde{Q}(z) = \sum_{i = 0}^{\infty} Q_{-m+i} (z- \beta)^{i}.
\end{equation}
The Taylor expansion of $\tilde{Q}$ around $\beta$ is
\begin{equation}
\label{eq26}
\tilde{Q}(z) = \sum_{i = 0}^{\infty} \frac{1}{i!} \tilde{Q}^{(i)}(\beta) (z- \beta)^{i}
\end{equation}
\[
= \sum_{i = 0}^{\infty} \frac{1}{i!} \left[(z- \beta)^{m} Q(z)\right]^{(i)}(\beta) (z- \beta)^{i}.
\]
From \eqref{eq24} and \eqref{eq26}, we obtain
\begin{equation}
\label{eq28}
Q_{-m+i} = \frac{1}{i!} \tilde{Q}^{(i)}(\beta)
= \frac{1}{i!} \left[(z- \beta)^{m} Q(z)\right]^{(i)}(\beta),
\quad i = 0, 1, 2, \ldots .
\end{equation}

\textbf{1.2.} Unlike scalar functions, for which all zeros and poles are characterized solely by their orders, operator-valued functions exhibit a more intricate structure: zeros are characterized by finite-dimensional spaces of vector functions, called root functions, whereas poles are characterized by finite-dimensional spaces of pole cancellation functions. For this purpose, we use the following definitions of a root function and related concepts, which are consistent with the corresponding definitions in \cite{GS, B1, BLu, GT}.

\begin{defi}\label{definition22}
Let $Q : \mathcal{D}(Q) \rightarrow \mathcal{L}(\mathcal{H})$ be an operator-valued function holomorphic on $\mathcal{D}(Q)$. Suppose that the vector function $\boldsymbol{\varphi}(z)$ satisfies the following conditions at some point $\alpha \in \mathbb{C}$:
\begin{enumerate}[(a)]
\item $\boldsymbol{\varphi}(\alpha) \ne 0$;
\item $(Q(z)\boldsymbol{\varphi}(z))^{(l)} \to 0$ as $z \to \alpha$, for $0 \le l \le m-1$, where $l, m-1 \in \mathbb{N} \cup \{0\}$.
\end{enumerate}
Then $\boldsymbol{\varphi}(z)$ is called a \textit{root function} of $Q(z)$ of order at least $m$ at the critical point (or zero) $\alpha$. If $m$ is the maximal number for which condition~(b) holds, then $\boldsymbol{\varphi}(z)$ is said to be a root function of \textit{exact order} (or \textit{exact partial multiplicity}) $m$ at $\alpha$.

The vector $\boldsymbol{\varphi}(\alpha)$ is called an \textit{eigenvector corresponding to the eigenvalue} $\alpha$. If $\boldsymbol{\varphi}(z)$ is a root function of maximal order $m$ among all root functions with the same eigenvector $\boldsymbol{\varphi}(\alpha)$, then $\boldsymbol{\varphi}(z)$ is called a \textit{canonical root function} associated with $\boldsymbol{\varphi}(\alpha)$. The sum of the orders of all linearly independent canonical root functions at $\alpha$ is called the (total) multiplicity of the zero at $\alpha$.

The maximal order among all canonical root functions at $\alpha$ is called the \textit{order of the zero} at $\alpha$.
\end{defi}

Pole cancellation functions are closely related to root functions. They were introduced and extensively studied in the investigation of an important subclass of operator-valued functions, namely generalized Nevanlinna functions; see, for example, \cite{DLS, BLu}.
\begin{defi}\label{definition24}
Let $Q : \mathcal{D}(Q) \rightarrow \mathcal{L}(\mathcal{H})$ be an operator-valued function holomorphic on $\mathcal{D}(Q)$. Suppose that the vector functions $\boldsymbol{\hat{\varphi}}(z)$ and $\boldsymbol{\psi}(z)$ satisfy the following conditions at a pole $\beta \in \mathbb{C}$ of $Q$:
\begin{enumerate}[(a)]
\item $\boldsymbol{\hat{\varphi}}(z) := Q(z)\boldsymbol{\psi}(z) \to c := \boldsymbol{\hat{\varphi}}(\beta) \ne 0$, $c \ne \infty$, as $z \to \beta$;
\item $\boldsymbol{\psi}^{(l)}(z) \to 0$ as $z \to \beta$, for $0 \le l \le m-1$, where $l, m-1 \in \mathbb{N} \cup \{0\}$.
\end{enumerate}
The functions $\boldsymbol{\hat{\varphi}}(z)$ and $\boldsymbol{\psi}(z)$ are called a \textit{pole function} and a \textit{pole cancellation function}, respectively, of $Q(z)$ at $\beta$ of order at least $m$. If $m$ is the maximal number for which condition~(b) holds for a given vector $\boldsymbol{\hat{\varphi}}(\beta)$, then $\boldsymbol{\hat{\varphi}}(z)$ and $\boldsymbol{\psi}(z)$ are called the \textit{pole function} and the \textit{pole cancellation function}, respectively, of $Q(z)$ at $\beta$ of \textit{exact order} $m$.

If $\boldsymbol{\psi}(z)$ is a pole cancellation function of maximal order $m$ among all pole cancellation functions associated with the same limit $\boldsymbol{\hat{\varphi}}(\beta)$, then $\boldsymbol{\psi}(z)$ is called the \textit{canonical pole cancellation function} associated with the vector $\boldsymbol{\hat{\varphi}}(\beta)$. The sum of the orders of all linearly independent canonical pole cancellation functions at $\beta$ is called the (total) multiplicity of the pole at $\beta$.

The maximal order among all canonical pole cancellation functions at $\beta$ coincides with the \textit{order of the pole} at $\beta$.
\end{defi}

By definition, $\beta = \infty$ \textit{is a pole} (respectively, a \textit{zero}) of $Q$ if and only if $0$ is a pole (respectively, a zero) of $G(\zeta) := Q(1/\zeta)$. Hence, all of the preceding definitions and considerations extend naturally to the pole $\beta = \infty$ of $Q$ via the function $G(\zeta)$ at $\zeta = 0$.

If an operator-valued function $Q(z)$ is boundedly invertible on some domain, that is, regular, and if $\boldsymbol{\psi}(z)$ is a pole cancellation function of $Q$ at the pole $\beta$ of order $k$, then
\begin{equation}
\label{eq210}
\boldsymbol{\hat{\varphi}}(z) := Q(z)\boldsymbol{\psi}(z)
\end{equation}
is a root function of $\hat{Q}(z) := Q(z)^{-1}$ at the zero $\beta$ of the same order $k$. In this way, all of the above definitions for a pole $\beta$ of $Q$ can be applied to a zero $\beta$ of $Q^{-1}$.

Therefore, if $Q(z)$ is a regular operator-valued function, there is a one-to-one correspondence between root functions $\boldsymbol{\hat{\varphi}}(z)$ of $Q^{-1}(z)$ and pole cancellation functions $\boldsymbol{\psi}(z)$ of $Q(z)$ at $\beta$.

Recall that pole cancellation functions were studied in a subclass of operator-valued functions, namely the generalized Nevanlinna class; see, for example, \cite{Lu3}. Since generalized Nevanlinna functions admit integral and operator representations, such representations have been employed in the construction of root and pole cancellation functions. In particular, pole cancellation functions have been characterized in terms of the Jordan vectors of the representing operators or relations of those generalized Nevanlinna functions. Subsequently, root and pole cancellation functions were used in the study of generalized zeros and poles of generalized Nevanlinna functions; see \cite{BLu}.

In this paper, we deal with the general class of operator-valued functions. Therefore, we cannot use operator or integral representations and must rely on other methods. Instead, we use Laurent and Taylor expansions.

Assume that $\beta \in \mathbb{C}$ is an ordinary pole of a regular operator-valued function $Q$. Then $Q$ admits the Laurent expansion \eqref{eq22} at $\beta$. The following question naturally arises: how can one find a canonical system of pole cancellation functions of $Q$ and the corresponding root functions at the zero $\beta$ of $Q^{-1}$ in terms of the operators appearing in the Laurent expansion \eqref{eq22}? In other words, for $k \le m$, we seek to construct a pole cancellation function of order $k$ and the corresponding root function of $Q^{-1}$, provided that these functions of the given order $k$ exist. To this end, we first establish a characterization of pole cancellation functions of $Q$ at a pole $\beta$; see Theorem~\ref{theorem42}, Corollary~\ref{corollary44}, and Proposition~\ref{proposition46}.

This type of material, which enables concrete applications, is best understood through examples. For this reason, we conclude by illustrating the theoretical results with an explicit example.
\section{Construction of pole cancellation functions}\label{s4} 
\begin{thm}\label{theorem42}
Let $Q : \mathcal{D}(Q) \rightarrow \mathcal{L}(\mathcal{H})$ be an operator-valued function holomorphic on $\mathcal{D}(Q)$, and let $\beta \in \mathbb{C}$ be a pole of $Q(z)$ of order $m \in \mathbb{N}$.

If a function $\boldsymbol{\psi}(z)$ is a pole cancellation function of $Q(z)$ at $\beta$ of order $k \in \mathbb{N}$, $k < m$, written in the form 
\begin{equation}\label{eq42}
\boldsymbol{\psi}(z) = \sum_{i=k}^{m+k-1} (z - \beta)^{i} \boldsymbol{\psi}_{i} +  O(z - \beta)^{m+k}, \quad \boldsymbol{\psi}_{k}\neq 0, 
\end{equation}
then it satisfies the conditions
\begin{equation}\label{eq44}
\sum_{i=k}^{m-s} Q_{-s-i} \boldsymbol{\psi}_{i} = 0, \quad s = m-k, m-k-1, \dots, 1,
\end{equation}
\begin{equation}\label{eq46}
\sum_{i=k}^{m} Q_{-i} \boldsymbol{\psi}_{i} \neq 0.
\end{equation}

Recall that $O(z - \beta)^{t} \in \mathcal{H}$, $t>0$, denotes a function satisfying, for some $M>0$,
\[
|O(z - \beta)^{t}| \leq M |z - \beta|^{t} \quad \text{as} \quad z \to \beta.
\]

Conversely, if the function $\boldsymbol{\psi}(z)$ given by \eqref{eq42} satisfies conditions \eqref{eq44} and \eqref{eq46}, then $\boldsymbol{\psi}(z)$ is a pole cancellation function of $Q(z)$ at $\beta$ of order $l, k\le l \le m$.
\end{thm}

\begin{proof}
Assume that $\beta \in \mathcal{D}(Q)$ is a pole of $Q(z)$ of order $m \in \mathbb{N}$, and that $\boldsymbol{\psi}(z)$ given by \eqref{eq42} is a pole cancellation function of $Q$ at $\beta$ of order $k < m$. According to \eqref{eq22} and \eqref{eq42}, in some neighborhood of $\beta$ excluding $\beta$, we have
\[
\boldsymbol{\hat{\varphi}}(z) := Q(z) \boldsymbol{\psi}(z) =
\]
\[
= \left( \sum_{j=1}^{m} \frac{Q_{-j}}{(z-\beta)^j} + Q_{0}(z) \right) 
\left( \sum_{i=k}^{m+k-1} (z - \beta)^{i} \boldsymbol{\psi}_{i} + O(z - \beta)^{m+k} \right) =
\]
\[
= \sum_{j=1}^{m} \sum_{i=k}^{m+k-1} \frac{Q_{-j} \boldsymbol{\psi}_{i}}{(z-\beta)^{j-i}} + O(z - \beta)^{k}.
\]

After the substitution $n = i - j$, we obtain
\[
\boldsymbol{\hat{\varphi}}(z) = \sum_{n=-m+k}^{-1} \left( \sum_{i=k}^{n+m} Q_{n-i} \boldsymbol{\psi}_{i} \right) (z - \beta)^{n} 
+ \sum_{n=0}^{k-1} \left( \sum_{i=k}^{n+m} Q_{n-i} \boldsymbol{\psi}_{i} \right) (z - \beta)^{n} + O(z - \beta)^{k}.
\]

After the substitution $s = -n$ in the first sum, we have
\begin{equation}\label{eq48}
\boldsymbol{\hat{\varphi}}(z) = \sum_{s=m-k}^{1} \left( \sum_{i=k}^{m-s} Q_{-s-i} \boldsymbol{\psi}_{i} \right) \frac{1}{(z - \beta)^{s}}
+ \sum_{n=0}^{k-1} \left( \sum_{i=k}^{n+m} Q_{n-i} \boldsymbol{\psi}_{i} \right) (z - \beta)^{n} + O(z - \beta)^{k}.
\end{equation}

Since $\boldsymbol{\psi}(z)$ is a pole cancellation function by assumption, $\boldsymbol{\hat{\varphi}}(z)$ is holomorphic in a neighborhood of $\beta$ with $\boldsymbol{\hat{\varphi}}(\beta) \neq 0$. Therefore,
\begin{equation}\label{eq410}
\boldsymbol{\hat{\varphi}}(z) = \boldsymbol{\hat{\varphi}}(\beta) + (z-\beta)\boldsymbol{\hat{\varphi}}'(\beta) + \cdots + \frac{1}{(k-1)!}(z-\beta)^{k-1} \boldsymbol{\hat{\varphi}}^{(k-1)}(\beta) + O(z-\beta)^{k}.
\end{equation}

Comparing \eqref{eq48} with \eqref{eq410} implies that the fractional part of \eqref{eq48} must vanish, and the constant term $\boldsymbol{\hat{\varphi}}(\beta) \neq 0$. Hence, for $k < m$, the conditions \eqref{eq44} and \eqref{eq46} must hold. 

Conversely, assume that the function $\boldsymbol{\psi}(z)$ given by \eqref{eq42} satisfies condition \eqref{eq44}. Then the right-hand side of \eqref{eq48} is holomorphic in a neighborhood of $\beta$. By definition, the function $\boldsymbol{\psi}(z)$ is therefore a pole cancellation function of $Q$ at $\beta$.

Moreover, by assumption, the function $\boldsymbol{\psi}(z)$ satisfies condition \eqref{eq46}. Hence, 
\[
\boldsymbol{\hat{\varphi}}(\beta) = \sum_{i=k}^{m} Q_{-i} \boldsymbol{\psi}_{i} \neq 0.
\]
This means that for at least one $l$, such that $k \le l \le m$, the vector $\boldsymbol{\psi}_{l} \neq 0 $. By definition, the function $\boldsymbol{\psi}(z)$ given by (\ref{eq42}) is of order $l$.  

This completes the proof of the theorem.
\end{proof} 

\begin{cor}\label{corollary44} Let $Q : \mathcal{D}(Q) \rightarrow \mathcal{L}(\mathcal{H})$ be an operator-valued function holomorphic on $\mathcal{D}(Q)$, and let $\beta \in \mathbb{C}$ be a pole of $Q(z)$ of order $m \in \mathbb{N}$.

The function $\boldsymbol{\psi}(z)$ given by \eqref{eq42} is a pole cancellation function of $Q(z)$ at $\beta$ of order $k = m$
if and only if it satisfies
\begin{equation}\label{eq412}
Q_{-m} \boldsymbol{\psi}_{m} \neq 0.
\end{equation}
\end{cor}

\begin{proof}
If $k = m$, then the fractional part of the product $\boldsymbol{\hat{\varphi}}(z) := Q(z)\boldsymbol{\psi}(z)$ does not exist. Therefore, the first sum in \eqref{eq48} does not appear, i.e., condition \eqref{eq44} is not applicable. Clearly, condition \eqref{eq46} reduces to \eqref{eq412}.
\end{proof}

\begin{propo}\label{proposition46}
Let $Q: \mathcal{D}(Q) \rightarrow \mathcal{L}(\mathcal{H})$ be an regular operator-valued function holomorphic on $\mathcal{D}(Q)$, and let $\beta \in \mathbb{C}$ be a pole of $Q(z)$ of order $m \in \mathbb{N}$. Let the function $\boldsymbol{\psi}(z)$ be a pole cancellation function of 
$Q(z)$ at $\beta$ of order $k \in \mathbb{N}, k\le m$, given in the form 
(\ref{eq42}). Then the corresponding root function of $Q^{-1}(z)$ at the zero $\beta$ is given by 
\begin{equation}\label{eq414} 
\boldsymbol{\hat{\varphi}}(z)=\boldsymbol{\hat{\varphi}}_{0}+(z-\beta)\boldsymbol{\hat{\varphi}}_{1}+...+(z-\beta)^{k-1}\boldsymbol{\hat{\varphi}}_{k-1}+O(z-\beta)^{k},
\end{equation}
where the vectors
\[
\boldsymbol{\hat{\varphi}}_{0},\boldsymbol{\hat{\varphi}}_{1},...,\boldsymbol{\hat{\varphi}}_{k-1},
\]
are given by
\begin{equation}\label{eq416} 
\boldsymbol{\hat{\varphi}}_{n}= \sum_{i=k}^{n+m}Q_{n-i}\psi_{i}, \quad n=0, 1,...,k-1.
\end{equation}
\end{propo}
\begin{proof} By assumption, the function $\boldsymbol{\psi}(z)$ is a pole cancellation function. Therefore, the first sum in \eqref{eq48} vanishes. The statement then follows directly from the second sum in \eqref{eq48}.
\end{proof} 

The coefficients $\boldsymbol{\hat{\varphi}}_{0}, \boldsymbol{\hat{\varphi}}_{1}, \ldots, \boldsymbol{\hat{\varphi}}_{k-1}$ of the root function (\ref{eq414}) will be called \textit{generalized Jordan vectors}. Recall that, in the special case where the operator-valued function $Q(z)$ is a matrix polynomial, the corresponding vectors are called generalized eigenvectors or Jordan vectors; see \cite[p. 24]{GLR}, \cite{B5}.

\begin{rem}\label{remark48}
If the function $\boldsymbol{\psi}(z)$ is a pole cancellation function of $Q(z)$ at $\beta$ of order $k \in \mathbb{N}$, $k < m$, then the coefficients $\boldsymbol{\psi}_{k}, \ldots, \boldsymbol{\psi}_{m}$ must satisfy conditions (\ref{eq44}) and (\ref{eq46}). These conditions will therefore be used to determine the coefficients $\boldsymbol{\psi}_{k}, \ldots, \boldsymbol{\psi}_{m}$. Subsequently, formula \eqref{eq416} will be applied to obtain the generalized Jordan vectors of the corresponding root function $\boldsymbol{\hat{\varphi}}(z)$.

According to the converse part of Theorem \ref{theorem42}, solving equations (\ref{eq44}) and (\ref{eq46}) does not guarantee that a pole cancellation function of exact order $k$ is obtained. Only if one can select a solution satisfying $\boldsymbol{\psi}_{k} \neq 0$ does one obtain a pole cancellation function of exact order $k$.

In the case $k = m$, condition \eqref{eq412} will be used to determine $\boldsymbol{\psi}_{m}\neq 0$.
\end{rem}

The following example illustrates the above theoretical considerations.
\begin{exm}
\label{example414}
Find a canonical set of pole cancellation functions at the pole $\beta = 1$ of the matrix-valued function
\[
Q(z)=\left( {\begin{array}{*{20}c}
\frac{1}{(z-1)^{2}} &  1 & z+1\\
1 &  z^{2} & z\\
z+1 &  z & \frac{1}{z-1}\\
\end{array} } \right).
\]
Also find the corresponding set of generalized Jordan vectors at the zero $\beta = 1$ of $Q^{-1}$.
\end{exm}

Observe that the pole cancellation function $\boldsymbol{\psi}(z)$ cannot be of order $k > m$, since in that case condition (a) in Definition \ref{definition24} would not be satisfied. This is why, in conditions (\ref{eq44}), (\ref{eq46}), (\ref{eq412}), and formula (\ref{eq416}), only the coefficients of the principal part of the Laurent expansion of $Q(z)$ around the pole $\beta = 1$ are used.

We determine the coefficients of the Laurent expansion by first finding the Taylor expansion (\ref{eq26}) of the function
\[
\tilde{Q}(z):=(z-1)^{2}Q(z)=\left( {\begin{array}{*{20}c}
1 & (z-1)^{2} & (z-1)^{2}(z+1)\\
(z-1)^{2} &  (z-1)^{2}z^{2} & (z-1)^{2}z\\
(z-1)^{2}(z+1) &  (z-1)^{2}z & (z-1)\\
\end{array} } \right)
\]
and then using the identity
\begin{equation}\label{eq418}
Q(z)=\frac{1}{(z-1)^{2}}\tilde{Q}(z),
\end{equation}
as explained in the Introduction. In this particular case, the first equation of identity (\ref{eq28}) takes the form
\begin{equation}\label{eq420}
Q_{j}=\frac{1}{(j+2)!}\tilde{Q}^{(j+2)}(1), \quad j=-2,-1,0,1,2.
\end{equation}
The principal part has only two coefficients, namely $Q_{j}$ for $j=-2,-1$. Therefore, we need only  $\tilde{Q}(1)$ and $\tilde{Q}^{(1)}(1)$.
\[
\tilde{Q}^{(1)}(z)=\left( {\begin{array}{*{20}c}
0 & 2(z-1) & (z-1)(3z+1)\\
2(z-1) &  2z(z-1)(2z-1) & (z-1)(3z-1)\\
(z-1)(3z+1) &  (z-1)(3z-1) & 1\\
\end{array} } \right).
\]
According (\ref{eq420}), we have
\[
Q_{-2}=\left( {\begin{array}{*{20}c}
1 & 0 & 0\\
0 &  0 & 0\\
0 &  0 & 0\\
\end{array} } \right), \quad
Q_{-1}=\left( {\begin{array}{*{20}c}
0 & 0 & 0\\
0 &  0 & 0\\
0 &  0 & 1\\
\end{array} } \right).
\]
Let us first find a pole cancellation function of order $k = 1$. Since $m = 2$, condition (\ref{eq44}) yields $s = m - k = 1$. Therefore, the sum in (\ref{eq44}) reduces to
\begin{equation}\label{eq422}
Q_{-2}\boldsymbol{\psi}_{1}=0 \Leftrightarrow \left( {\begin{array}{*{20}c}
1 & 0 & 0\\
0 &  0 & 0\\
0 &  0 &0\\
\end{array} } \right)\left( {\begin{array}{*{20}c}
\psi_{1}^{1} \\
\psi_{2}^{1} \\
\psi_{3}^{1}\\
\end{array} } \right)=\left( {\begin{array}{*{20}c}
0 \\
0 \\
0\\
\end{array} } \right). 
\end{equation}
\[
\Rightarrow\boldsymbol{\psi}_{1}=\left( {\begin{array}{*{20}c}
0 \\
\psi_{2}^{1} \\
\psi_{3}^{1}\\
\end{array} } \right), \quad \psi_{2}^{1}, \psi_{3}^{1} \in \mathbb{C}.
\]
Conditions (\ref{eq46}) yields
\[
Q_{-1}\boldsymbol{\psi}_{1}+ Q_{-2}\boldsymbol{\psi}_{2}= \left( {\begin{array}{*{20}c}
0 & 0 & 0\\
0 &  0 & 0\\
0 &  0 & 1\\
\end{array} } \right)\left( {\begin{array}{*{20}c}
0 \\
\psi_{2}^{1} \\
\psi_{3}^{1}\\
\end{array} } \right)+\left( {\begin{array}{*{20}c}
1 & 0 & 0\\
0 &  0 & 0\\
0 &  0 & 0\\
\end{array} } \right)\left( {\begin{array}{*{20}c}
\psi_{1}^{2} \\
\psi_{2}^{2} \\
\psi_{3}^{2}\\
\end{array} } \right) \neq 0.
\]
This condition is satisfied whenever 
\begin{equation}\label{eq424}
\left( {\begin{array}{*{20}c}
\psi_{1}^{2} \\
0 \\
\psi_{3}^{1}\\
\end{array} } \right) \neq 0.
\end{equation}
Therefore, a pole cancellation function of order $k=1$ or $k=2$ is given by
\[
\boldsymbol{\psi}(z)=\left( {\begin{array}{*{20}c}
0 \\
\psi_{2}^{1} \\
\psi_{3}^{1}\\
\end{array} } \right)(z-1)+\left( {\begin{array}{*{20}c}
\psi_{1}^{2} \\
\psi_{2}^{2} \\
\psi_{3}^{2}\\
\end{array} } \right) (z-1)^{2}+O(z-1)^{3},
\]
provided that the condition (\ref{eq424}) is satisfied. 

Evidently, if 
\[
\psi_{3}^{1} \neq 0 \vee \left( \psi_{2}^{1} \neq 0 \wedge \psi_{1}^{2} \neq 0\right) 
\]
the function $\boldsymbol{\psi}(z)$ is a pole cancellation function of first order.

It remains to determine the corresponding root function $\boldsymbol{\hat{\varphi}}(z)$ at the zero $\beta=1$ of $Q^{-1}(z)$. This root function must be of first order and therefore consists only of the eigenvector $\boldsymbol{\hat{\varphi}}_{0}$, which is computed from (\ref{eq416}) for $m=2$, $k=1$, and $n=0$. Hence,
\[
\boldsymbol{\hat{\varphi}}(z)=\boldsymbol{\hat{\varphi}}_{0}=Q_{-1}\boldsymbol{\psi}_{1}+Q_{-2}\boldsymbol{\psi}_{2}=
\] 
\[
=\left( {\begin{array}{*{20}c}
0 & 0 & 0\\
0 &  0 & 0\\
0 &  0 & 1\\
\end{array} } \right)\left( {\begin{array}{*{20}c}
0 \\
\psi_{2}^{1} \\
\psi_{3}^{1}\\
\end{array} } \right)+\left( {\begin{array}{*{20}c}
1 & 0 & 0\\
0 &  0 & 0\\
0 &  0 & 0\\
\end{array} } \right)\left( {\begin{array}{*{20}c}
\psi_{1}^{2} \\
\psi_{2}^{2} \\
\psi_{3}^{2}\\
\end{array} } \right)=\left( {\begin{array}{*{20}c}
\psi_{1}^{2} \\
0 \\
\psi_{3}^{1}\\
\end{array} } \right),
\]
provided that conditions (\ref{eq424}) is satisfied.

Since the task is to find a canonical set of pole cancellation functions of $Q$ at $\beta$ and the corresponding root functions of $Q^{-1}$, we now determine these functions for $k=m=2$. According to (\ref{eq42}), we have
\begin{equation}\label{eq426}
\boldsymbol{\psi}(z) =(z - 1)^{2} \boldsymbol{\psi}_{2} + (z - 1)^{3} \boldsymbol{\psi}_{3}+ O(z - 1)^{4}, \quad \boldsymbol{\psi}_{2}\neq 0.
\end{equation}

Since $k=m$, according to Corollary \ref{corollary44}, conditions (\ref{eq44}) and (\ref{eq46}) reduce to condition (\ref{eq412}) alone,
\[
Q_{-2}\boldsymbol{\psi}_{2}\neq 0.
\]
\[
\Rightarrow \left( {\begin{array}{*{20}c}
1 & 0 & 0\\
0 &  0 & 0\\
0 &  0 & 0\\
\end{array} } \right)\left( {\begin{array}{*{20}c}
\psi_{1}^{2} \\
\psi_{2}^{2}\\
\psi_{3}^{2}\\
\end{array} } \right) \neq 0.\Rightarrow \psi_{1}^{2} \neq 0.
\]
According to Corollary \ref{corollary44}, every function of the form (\ref{eq426}) with $\psi_{1}^{2} \neq 0$, and with the vector $\boldsymbol{\psi}_{3}$ chosen arbitrarily, is a pole cancellation function of order $k=2$.

Since $\boldsymbol{\psi}(z)$ is a pole cancellation function of order $k=2$, the corresponding root function $\boldsymbol{\hat{\varphi}}(z)$ of $Q^{-1}$ is also of order $k=2$. According to (\ref{eq416}),
\[
\boldsymbol{\hat{\varphi}}_{0}=Q_{-2}\psi_{2}=\left( {\begin{array}{*{20}c}
\psi_{1}^{2} \\
0\\
0\\
\end{array} } \right),
\]
\[
\boldsymbol{\hat{\varphi}}_{1}=Q_{-1}\psi_{2}+Q_{-2}\psi_{3}=\left( {\begin{array}{*{20}c}
0 \\
0 \\
\psi_{3}^{2}\\
\end{array} } \right)+ \left( {\begin{array}{*{20}c}
\psi_{1}^{3} \\
0 \\
0\\
\end{array} } \right)=\left( {\begin{array}{*{20}c}
\psi_{1}^{3}  \\
0 \\
\psi_{3}^{2}\\
\end{array} } \right).
\] 
This implies that the root function at the zero $\beta=1$ of the function $Q^{-1}(z)$ is
\[
\boldsymbol{\hat{\varphi}}(z)=\left( {\begin{array}{*{20}c}
\psi_{1}^{2} \\
0\\
0\\
\end{array} } \right)+\left( {\begin{array}{*{20}c}
\psi_{1}^{3} \\
0 \\
\psi_{3}^{2}\\
\end{array} } \right)(z-1), \quad \psi_{1}^{2}\neq 0.
\] 
\hfill $\square$







\vspace{.5cm}
\begin{footnotesize}
	\begin{tabular}{l}
Muhamed Borogovac\\
Boston Mutual Life\\
Actuarial Department\\
120 Royall St. Canton, MA 02021, USA\\
e-mail: {muhamed.borogovac@gmail.com }\\
\end{tabular}
\end{footnotesize}
\label{LastPage}
\end{document}